\documentclass[12pt,oneside]{amsart}

\usepackage{geometry}  
\usepackage{enumerate}
\usepackage{url}
\usepackage{todonotes}
\usepackage{color}
\usepackage{mathrsfs}
\usepackage{bm}
\usepackage{bbm}
\usepackage{wasysym}
\usepackage[utf8]{inputenc}
\inputencoding{utf8}
\usepackage{comment}
\usepackage{graphicx}
\usepackage[justification=centering]{caption}
\usepackage{amsthm,amssymb,amsmath}
\usepackage{subfigure}
\usepackage{hyperref}
\allowdisplaybreaks[4]

\newtheorem{conjecture}{Conjecture}

\newtheorem{theorem}{Theorem}[section]

\newtheorem{lemma}[theorem]{Lemma}

\newtheorem{corollary}[theorem]{Corollary}

\theoremstyle{definition}
\newtheorem{definition}[theorem]{Definition}

\title{A note on cube-free problems}
\author{Yuchen Meng}
\date{}

\begin{document}

\begin{abstract}
    Eberhard and Pohoata conjectured that every $3$-cube-free subset of $[N]$ has size less than $2N/3+o(N)$. In this paper we show that if we replace $[N]$ with $\mathbb{Z}_N$ the upper bound of $2N/3$ holds, and the bound is tight when $N$ is divisible by $3$ since we have $A=\{a\in \mathbb{Z}_N:a\equiv 1,2\pmod{3}\}.$ Inspired by this observation we conjecture that every $d$-cube-free subset of $\mathbb{Z}_N$ has size less than $(d-1)N/d$ where $N$ is divisible by $d$, and we show the tightness of this bound by providing an example $B=\{b\in\mathbb{Z}_N:b\equiv 1,2,\ldots,d-1\pmod{d}\}$. We prove the conjecture for several interesting cases, including when $d$ is the smallest prime factor of $N$, or when $N$ is a prime power. 

   We also discuss some related issues regarding $\{x,dx\}$-free sets and $\{x,2x,\ldots,dx\}$-free sets. A main ingredient we apply is to arrange all the integers into some square matrix, with $m=d^s\times l$ having the coordinate $(s+1,l-\lfloor l/d\rfloor)$. Here $d$ is a given integer and $l$ is not divisible by $d$.
\end{abstract}

\maketitle

\section{Introduction}
A set is called sum-free if there are no solutions to the equation $x+y=z$. For example, any subset of integers consisting of odd numbers is sum-free, as the sum of any two odd numbers results in an even number. The study on sum-free set traces its roots back to the early 20th century when Schur~\cite{Schur1} used a combinatorial argument to show that Fermat's Last Theorem does not hold in the finite field $\mathbb{F}_p$.

In addition to exploring the sum-free problem, the research community has shown considerable interest in generalizations. For instance, one of them is to study the so-called  $(k,l)$-sum-free sets, which is a set with no solutions to $x_1+\dots+x_k=y_1+\dots+y_l$. In particular, the avoidance density for such sets was recently determined by
by Jing and Wu~\cite{JW1,JW2}, generalizing the line of research for sum-free sets by Bourgain~\cite{Bourgain97}, by Eberhard, Green, and Manners~\cite{EGM}, and by Eberhard~\cite{Eberhard15}. 

In this note our primary focus lies in yet another branch of generalization for sum-free problems --- the study of cube-free subsets within the cyclic group $\mathbb{Z}_N$. To establish the foundation for our exploration, we present the definition of cubes, or more precisely, projective cubes:

\begin{definition}
    Given a multiset $S=\{a_1,\ldots,a_d\}$ of size $d$, we define
    the projective $d$-cube generated by $S$ as
    $$\Sigma^*S=\left\{\sum_{i\in I}a_i:\ \varnothing\neq I\subset [d]\right\}.$$
\end{definition}

\begin{definition}
    We say $A$ is $d$-cube-free if there does not exist a multiset $S$ of size $d$ with $(\Sigma^*S)\subset A$.
\end{definition}

For example, a set is $3$-cube-free if it contains no $\{x,y,z,x+y,y+z,x+z,x+y+z\}$ as a subset.

The motivation behind this research is derived from a similar problem concerning cube-free subsets of the set $[N]$, which was conjectured  by Eberhard and Pohoata:

\begin{conjecture}[Eberhard--Pohoata]
    Suppose $A\subset [N]$ is $3$-cube-free, then 
    $$|A|\leqslant (2/3+o(1))N.$$
    The equality holds when $A=\{x\equiv 1,2 \pmod{3}\}$ or $A=(N/3,N]$. 
    
\end{conjecture}

It is easy to verify that the two examples are $3$-cube-free. However, it is important to note that when discussing the problem within cyclic groups whose order is divisible by $3$, the latter condition is no longer $3$-cube-free, while the former still holds. This observation suggests the following conjecture:

\begin{conjecture}\label{conj d-cube}
    Let $A$ be a $d$-cube-free subset of $\mathbb{Z}_{N}$ where $d\mid N$, then
    $$|A|\leqslant \dfrac{d-1}{d}N.$$
\end{conjecture}

The main result of this note verifies Conjecture~\ref{conj d-cube} for many interesting cases:

\begin{theorem}\label{main thm}
    Let $A$ be a $d$-cube-free subset of $\mathbb{Z}_{N}$ where $d\mid N$. We have
    $$|A|\leqslant \dfrac{d-1}{d}N$$
    when one of following is true:\medskip
    
    \noindent\emph{(i)} $d=3$. \medskip
    
    \noindent\emph{(ii)} $d$ is the smallest prime factor of $N$. \medskip
    
    \noindent\emph{(iii)} $N$ is the power of some prime $p$. \medskip
\end{theorem}

Notably, the method we employed in proving Theorem~\ref{main thm} (i) holds considerable promise for  addressing  similar problems. To establish this theorem, we concentrate on subsets that are free of the diagonal solutions, namely $\{x, 2x, \ldots, (d-1)x\}$-free, and the proofs for Theorem~\ref{main thm} (ii) and Theorem~\ref{main thm} (iii) are subsequently derived from this fundamental idea.

\section{The tightness of the upper bound}\label{tightness}

\begin{theorem}
    The bound in conjecture \ref{conj d-cube} is tight, since we have 
    $$A=\{a\in\mathbb{Z}_{N}:a\equiv 1,2,\ldots,d-1\pmod{d}\}$$
    which is $d$-cube-free.
\end{theorem}

We are going to prove a lemma to show the correctness of the example, which is based on the Cauchy--Davenport Theorem~\cite{Cauchy1,Davenport1}. Throughout this section we use standard definitions and notations in Additive Combinatorics as given in~\cite{TV1}. Given $A,B\subset\mathbb{Z}$, we write
$$A+B:=\{a+b:a\in A,b\in B\}, \quad\text{and}\quad AB:=\{ab: a\in A,b\in B\}.$$
When $A=\{x\}$, we simply write $x+B:=\{x\}+B$ and $x\cdot B:=\{x\}B$. 

\begin{theorem}[Cauchy--Davenport]
Let $A,B\subset \mathbb{F}_p$, then
$$|A+B|\geqslant min\{|A|+|B|-1,p\}.$$
\end{theorem}

\begin{lemma}
    Let $a_i\in\mathbb{Z}_d\backslash \{0\}$, $\lambda_i\in\{0,1\}$. We define
    $$S_t:=\left\{\sum_{i=1}^t\lambda_ia_i:(\lambda_1,\lambda_2,\ldots,\lambda_t)\neq(0,0,\ldots,0)\right\}$$
    where $t\leqslant d$. If $0\notin S_t$, then $|S_t|\geqslant t$. 
\end{lemma}

\begin{proof} 
    When $d=p$ is a prime, the lemma is indicated by the Cauchy--Davenport Theorem. Indeed, note that
    $$S_t\supseteq\{a_1\}+\sum_{i=2}^t\{0,a_i\}.$$
    The right hand side is the sum of $k$ sets. We only need to consider the case that $|\{a_1\}+\sum_{i=2}^k\{0,a_i\}|$ is strictly less than $d$ for all $k\leqslant t$, or else $|S_t|=d$ so that $0\in S_t$. By using the Cauchy--Davenport Theorem for $(t-1)$ times, we have
    $$|\{a_1\}+\sum_{i=2}^t\{0,a_i\}|\geqslant|\{a_1\}|+\sum_{i=2}^t|\{0,a_i\}|-(t-1).$$
    Thus $|S_t|\geqslant t.$

    When $d$ is not a prime, the proof goes by induction on $d$ and then induction on $t$. According to the discussion above, we have already proved the lemma for prime factors, as a foundation of the induction on $d$. Now we suppose that the lemma holds in $\mathbb{Z}_k$ with $k$ being all the factors of $d$ and start our induction on $t$. To begin with, $|S_t|\geqslant t$ for $t=1,2$. Indeed, when $t=2$, $S_t=\{a_1,a_2,a_1+a_2\}.$ It is impossible to have $a_1=a_2=a_1+a_2$, which implies that $a_1=a_2=0$. Now we assume $|S_k|\geqslant k$ for all integers $k\leqslant t$. Note that
    $$S_{t+1}\supseteq S_t+\{0,a_{t+1}\}\supseteq S_t.$$
    The induction hypothesis gives $|S_{t+1}|\geqslant |S_t|\geqslant t$. Suppose $|S_{t+1}|=t<t+1$, then
    $$S_{t+1}=S_t=S_t+a_{t+1}.$$
    Adding up all the elements in three sets respectively, we have
    $$\sum_{x\in S_{t+1}}x=\sum_{x\in S_{t}}x=\sum_{x\in S_{t}}(x+a_{t+1}).$$
    The second equality implies 
    $$ta_{t+1}\equiv 0\pmod{d}.$$
    Similarly by symmetry we have
    $$ta_j\equiv 0\pmod{d}$$
    for $j=1,2,\ldots,t+1.$
    Let $(t,d)=s,t=st_1,d=sd_1$, then $(t_1,d_1)=1$ and $d_1\mid a_j.$ We must have $s>1$, or else $d\mid a_j$, i.e. $a_j=0$. Now we have
    $$a_i'\not\equiv 0\pmod{s},$$
    $$S_{t+1}'=\left\{\sum_{i=1}^{t+1}\lambda_ia_i':(\lambda_1,\lambda_2,\ldots,\lambda_t)\neq(0,0,\ldots,0)\right\}.$$
    where $s=d/d_1<d,a_i'=a_i/d_i$. Clearly $|S_{t+1}'|=|S_{t+1}|=t$, so by induction hypothesis there exist $\lambda_1,\lambda_2,\ldots,\lambda_{t+1}$ not all zero such that
    $$\sum_{i=1}^{t+1}\lambda_ia_i'\equiv 0\pmod{s}.$$
    Hence
    $$\sum_{i=1}^{t+1}\lambda_ia_i\equiv 0\pmod{d}.$$
    Now we finish the proof by showing that $0\in S_{t+1}$ if $|S_{t+1}|\leqslant t$.
\end{proof}

\section{Related problems}

To start further discussion, we prove the case $d=3$ in advance. Recall Theorem \ref{main thm} (i).

\begin{theorem}
        Let $A$ be a 3-cube-free subset of $\mathbb{Z}_{N}$ where $3\mid N$, then
    $$|A|\leqslant \dfrac{2}{3}N.$$
\end{theorem}

\begin{proof}
    The proof consists of two parts discussing whether $A$ contains $\{x,2x\}$ as a subset for some $x$, namely $\{x,2x\}$-free or not. When $A$ is $\{x,2x\}$-free, it is equivalent to 
    $$A\cap 2\cdot A=\varnothing.$$
    Here $2\cdot A$ is defined by $2\cdot A:=\{2a:a\in A\}$, as mentioned in Section \ref{tightness}. Note that for all $a\in \mathbb{Z}_N$, there is at most one pair $(b,c)$ with $b\neq c$ such that $2b=2c=a$, which implies 
    $$|2\cdot A|\geqslant \dfrac{1}{2}|A|.$$
    Thus
    $$N\geqslant |A|+|2\cdot A|\geqslant \dfrac{3}{2}|A|,$$
    and then $A\leqslant 2N/3.$
    
    Now let $A$ be not $\{x,2x\}$-free, then there exists $x$ such that $x,2x\in A$. Consider the cube generated by $\{x,x,y\}$ where $y$ is selected among all the elements in $A$, we have
    $$A\cap (A-x)\cap (A-2x)=\varnothing.$$
    By taking the complementary set 
    $$A^c\cup (A-x)^c\cup (A-2x)^c=\mathbb{Z}_N.$$
    Note that both $(A-x)$ and $(A-2x)$ are copies of $A$, thus
    $$3|A^c|\geqslant N,$$
    and then $|A|\leqslant 2N/3.$
\end{proof}

Actually the proof above can be generalized to all cyclic group, not necessarily $3\mid N$. It gives a quite trivial upper bound of $2N/3$, but in some cases there might exist a better one, for instance $(5/8+o(1))N$ conjectured by Long and Wagner~\cite{LW1} where $N=2^k.$

Inspired by the proof on $3$-cube discussing whether $\{x,2x\}$ is forbidden, one can naturally expect to generalize the proof to larger cubes, which leads to the conjecture as follows.

\begin{conjecture}\label{conj set-free}
    Let $A$ be a $\{x,2x,\ldots,(d-1)x\}$-free subset of $\mathbb{Z}_N$ where $d\mid N$, then
    $$|A|\leqslant \dfrac{d-1}{d}N.$$
\end{conjecture}

\noindent\emph{Proof of Conj \ref{conj d-cube} assuming Conj \ref{conj set-free}.} It suffices to prove it when $A$ is not $\{x,2x,\ldots,(d-1)x\}$-free. Now there must be an $x$ such that $x,2x,\ldots,(d-1)x\in A$. Consider the $d$-cube generated by $\{x,x,\ldots,x,y\}$ where $y$ is selected among all the elements in $A$, we have
$$A\cap (A-x)\cap (A-2x)\cap\cdots\cap (A-(d-1)x)=\varnothing.$$
This implies
$$|A|\leqslant \dfrac{d-1}{d}N.$$
\qed

It must be pointed out that we have a similar bound for $\{x,dx\}$-free subsets as follows.

\begin{theorem}\label{double-free}
    Let $A$ be a $\{x,dx\}$-free subset of $[N]$, then 
    $$|A|\leqslant\dfrac{d}{d+1}N+O(\log N).$$
\end{theorem}

\begin{proof}
    Given $d$, note that every positive integer $m$ can be uniquely written as $m=d^s\times l$ with $s$ being a non-negative integer and $l$ being a positive integer not divisible by $d$. Thus we can divide all the integers into different chains starting with integers not divisible by $d$: $l,dl,d^2l,\ldots$. We denote the chain started with $l$ by $C_l$.
    
    It is clear that $A$ is $\{x,dx\}$-free if and only if there are no two elements of $A$ adjacent in one chain. To acquire the upper bound, we just need to consider the extreme cases on different chains independently. Given $l$ and $C_l$, since only one of $\{d^kl,d^{k+1}l\}$ can be contained in $A$ for all $k\geqslant 0$ such that $d^{k+1}l\leqslant N$, the extreme case appears when the elements are selected alternately. More precisely, when $|[N]\cap C_l|$ is odd, the elements of $A\cap C_l$ take up all the odd positions in $C_l$; when $|[N]\cap C_l|$ is even, the elements of $A\cap C_l$ take up either all the odd positions or all the even positions in $C_l$. 
    
    Since different chains have different lengths, it is difficult to count $|A\cap C_l|$ respectively and then add them together. Instead, we count them by layers which are defined by
    $$L_i:=\{x\in\mathbb{Z}_{+}: d^{i-1}\mid x,\ d^i\nmid x\}.$$
    It is clear that all the integers can be divided into different layers, i.e.
    $$\mathbb{Z}_+=\bigcup_{i=1}^\infty L_{i}.$$
    
    For convenience, we may assume the elements of $A$ take up all the odd positions in $C_l$ no matter whether $|[N]\cap C_l|$ is even or odd, as it does not change the size. Based on this assumption, the maximal $A$ can be write as  
    $$A=\left(\bigcup_{i=0}^\infty L_{2i+1}\right)\cap [N].$$
    Also all the layers are pairwise disjoint, thus
    $$|A|=\sum_{i=0}^\infty|L_{2i+1}\cap[N]|.$$
    Suppose $d^s\leqslant N<d^{s+1}$. When $s$ is odd, we have
    \begin{align*}
        |A| &= \sum_{i=0}^{(s-1)/2}|L_{2i+1}\cap[N]| \\
        &= \sum_{i=0}^{(s-1)/2}\lfloor\dfrac{d-1}{d^{2i+1}}N+\dfrac{d-1}{d}\rfloor \\
        &= \sum_{i=0}^{(s-1)/2}\left(\dfrac{d-1}{d^{2i+1}}N+\dfrac{d-1}{d}\right)+O(s) \\
        &= \dfrac{d}{d+1}(1-\dfrac{1}{d^{s+1}})N+\dfrac{(s+1)(d-1)}{2d}+O(s) \\
        &= \dfrac{d}{d+1}N+O(s)=\dfrac{d}{d+1}N+O(\log N).
    \end{align*}
    And when $s$ is even, we have
    \begin{align*}
        |A| &= \sum_{i=0}^{s/2}|L_{2i+1}\cap[N]| \\
        &= \sum_{i=0}^{s/2}\lfloor\dfrac{d-1}{d^{2i+1}}N+\dfrac{d-1}{d}\rfloor \\
        &= \sum_{i=0}^{s/2}\left(\dfrac{d-1}{d^{2i+1}}N+\dfrac{d-1}{d}\right)+O(s) \\
        &= \dfrac{d}{d+1}(1-\dfrac{1}{d^{s+2}})N+\dfrac{(s+2)(d-1)}{2d}+O(s) \\
        &= \dfrac{d}{d+1}N+O(s)=\dfrac{d}{d+1}N+O(\log N).
    \end{align*}
\end{proof} 

This bound may be helpful when we consider cube-free subsets of $[N]$. It is worth noticing that the remainder term $O(\log N)$ cannot be removed. Indeed, take the case $d=2$ as an example and we have the following result.

\begin{theorem}\footnote{The author appreciates Sean Eberhard for providing the idea of comparing $N$ with $4N$, which helped give rise to the construction in the following proof.}
    Let $A$ be an $\{x,2x\}$-free subset of $[N]$, then 
    $$|A|\leqslant \dfrac{2}{3}N+O(\log N).$$
    Moreover, there exists $\varepsilon>0$, such that for all $n>0$, there exists $N>n$ such that $[N]$ contains a $\{x,2x\}$-free subset with size at least $2N/3+\varepsilon\log N$. 
\end{theorem}

\begin{proof}
    Take $d=2$ in Theorem~\ref{double-free} then we have $|A|\leqslant 2N/{3}+O(\log N).$ Now consider the sequence $a_n=(4^n-1)/3$ which converges to infinity, we are going to show that $D(a_n)=2a_n/3+n/3$, so that $\{a_n\}$ is the sequence we want. Here $D(N)$ is the size of the largest $\{x,2x\}$-free subsets of $[N]$.

    The proof goes by induction. Recall that 
    $$D(N)=\sum_{k=0}^\infty\left\lfloor{\dfrac{N+4^k}{2\cdot4^k}}
    \right\rfloor.$$
    Since $a_{n+1}=4a_n+1$, we have
    \begin{align*}
        D(a_{n+1}) &= \sum_{k=0}^\infty\left\lfloor{\dfrac{a_{n+1}+4^k}{2\cdot4^k}}\right\rfloor \\
        &= \sum_{k=0}^\infty\left\lfloor{\dfrac{4a_n+1+4^k}{2\cdot4^k}}
    \right\rfloor \\
        &= \sum_{k=1}^\infty\left\lfloor{\dfrac{4a_n+4^k}{2\cdot4^k}}
    \right\rfloor +2a_n+1 \\
        &= \sum_{k=0}^\infty\left\lfloor{\dfrac{a_n+4^k}{2\cdot4^k}}
    \right\rfloor+2a_n+1 \\
        &= D(a_n)+2a_n+1
    \end{align*}
    It remains to show that $D(1)=1$, which is clear.
\end{proof}

As for cyclic group case, we can get rid of the remainder term.

\begin{theorem}
    Let $A$ be a $\{x,dx\}$-free subset of $\mathbb{Z}_N$ and $k=(d,N)$, then 
    $$|A|\leqslant\dfrac{k}{k+1}N.$$
\end{theorem}

\noindent\emph{Proof.} We are going to count the number of solutions to the equation $x_0=da$ where $x_0$ is fixed. Suppose $da\equiv db\pmod{N}$, then
$$a\equiv b\pmod{N/k}.$$
This implies there will be at most $k$ solutions to the equation. Thus 
$$|d\cdot A|\geqslant\dfrac{1}{k}|A|.$$
Since $A\cap d\cdot A=\varnothing$, we have
$$|A|\leqslant\dfrac{k}{k+1}N.$$
\qed

\begin{corollary}\label{set-free 2}
    Let $A$ be a $\{x,(d-1)x\}$-free subset of $[N]$ where $d\mid N$, then $k=(d-1,N)\leqslant d-1$ and
    $$|A|\leqslant\dfrac{k}{k+1}N\leqslant\dfrac{d-1}{d}N.$$
\end{corollary}

Since $\{x,(d-1)x\}$ is a subset of $\{x,2x,\ldots,(d-1)x\}$, a larger density is implied when the latter is forbidden. But comparing Conjecture \ref{conj set-free} with Corollary \ref{set-free 2}, we find that these two sets give rise to a same density when forbidden(or at least we expect them to).

\section{Specific cases}

In this section we are going to prove Theorem \ref{main thm} (ii) and Theorem \ref{main thm} (iii) by showing Conjecture \ref{conj set-free} holds respectively. It must be pointed out that the idea of counting the family of sets partly comes from Long and Wagner~\cite{LW1}.

\subsection{When \texorpdfstring{$d$}{} is the smallest prime factor}

We define the set $\mathcal{F}$ as
$$\mathcal{F}:=\{\{x,2x,3x,\ldots,(d-1)x\}:x\in \mathbb{Z}_{N}\backslash\{0\}\}.$$

Note that every element $B$ in $\mathcal{F}$ has size exactly $d-1$. Otherwise, there exist $i_1,i_2\in [d-1]$ and $x\in \mathbb{Z}_{N}\backslash\{0\}$ such that
$$N\mid (i_1-i_2)x.$$
Since $|i_1-i_2|\leqslant d-2$ and $N$ has no prime factors smaller than $d$, we have
$$(i_1-i_2,N)=1$$
and then $N\mid x$ which is contradictory.

Moreover, observe that every element in $\mathbb{Z}_{N}\backslash\{0\}$ appears precisely $d-1$ times among all different $B$. Indeed, for all $x_0\in\mathbb{Z}_{N}\backslash\{0\}$ and $t\in [d-1]$, the congruence equation with respect to $x$
$$x_0\equiv tx\pmod{N}$$
has one and only solution. This is because $(t,N)=1$ and thus $0,t,2t,\ldots,(N-1)t$ form a complete system of residues modulo $N$.

Now we are able to figure out the size of $\mathcal{F}$ by double counting all the elements covered.
$$(d-1)|\mathcal{F}|=(d-1)(N-1).$$

Let $A$ be a $\{x,2x,\ldots,(d-1)x\}$-free subset, it is clear that for every $B\in \mathcal{F}$ there exist at least one element $a_B\in A^c$, and $a_B$ appears repeatedly at most $d-1$ times, which indicates
$$|A^c|\geqslant \dfrac{|\mathcal{F}|}{d-1}=\dfrac{N-1}{d-1}\geqslant \dfrac{N}{d}.$$
Then 
$$|A|\leqslant N-\dfrac{N}{d}=\dfrac{d-1}{d}N.$$

\subsection{When \texorpdfstring{$N=p^l$}{}}

We are to prove a better result for an $\{x,2x,\ldots,(p^d-1)x\}$-free subset $A$:
$$|A|\leqslant (1-\dfrac{1}{p^d-1})N.$$

First we define the layers in $\mathbb{Z}_{p^l}$ by
$$L_i:=\{x\in \mathbb{Z}_{p^l}: p^{i-1}\mid x,{p^i}\nmid x\}.$$
For convenience we write 
$$L_{[a,b]}:=L_a\cup L_{a+1}\cup\cdots\cup L_{b-1}\cup L_b.$$

The proof goes by dividing $\mathbb{Z}_{p^l}$ into $\lceil(l+1)/d\rceil$ blocks, each block consisting of several continuous layers. Indeed, with $q$ being the largest integer such that $qd\leqslant l+1$, the division is
$$\mathbb{Z}_{q^l}=L_{[1,d]}\cup L_{[d+1,2d]}\cup\cdots\cup L_{[(q-1)d+1,qd]}\cup L_{[qd+1,l+1]}.$$

For an integer $a\leqslant l-d+1$, we define the set $\mathcal{F}_a$ as
$$\mathcal{F}_a:=\{\{x,2x,3x,\ldots,(p^d-1)x\}:x\in L_a\}.$$

Note that for any $B\in \mathcal{F}_a$, $B$ has size exactly $p^d-1$. Otherwise there exists $i_1,i_2\in [p^d-1]$ and $x\in L_a$ such that
$$p^l\mid (i_1-i_2)x.$$
Since $|i_1-i_2|\leqslant p^d-2$, $(i_1-i_2)$ is not divisible by $p^d$. Recall that $x\in L_a$ with $a\leqslant l-d+1$, thus we can find a contradiction.

Moreover, observe that for every $B\in \mathcal{F}_a$, $B\subset L_{[a,a+d-1]}$ and
every element of $L_{[a,a+d-1]}$ appears in precisely $(p-1)p^{d-1}$(that is, the size of $[p^d-1]\cap L_1$) different sets in $\mathcal{F}_a$. 

Now we are able to figure out the size of $\mathcal{F}_a$ by double counting all the elements covered.
$$(p^d-1)|\mathcal{F}_a|=(p-1)p^{d-1}|L_{[a,a+d-1]}|.$$

Let $A$ be an $\{x,2x,\ldots,(p^d-1)x\}$-free subset. It is clear that for every set $B\in \mathcal{F}_a$ there exists an element $x_B\in B$ with $x_B\notin A$, and every $x_B$ recurs at most $(p-1)p^{d-1}$ times, therefore
$$|A^c\cap L_{[a,a+d-1]}|\geqslant \dfrac{|\mathcal{F}_a|}{(p-1)p^{d-1}}=\dfrac{1}{p^d-1}|L_{[a,a+d-1]}|.$$
$$\dfrac{|A\cap L_{[a,a+d-1]}|}{|L_{[a,a+d-1]}|}\leqslant 1-\dfrac{1}{p^d-1}.$$
Especially, we set $a=1+td,t=0,1,\ldots,q-1$ to obtain
\begin{equation}\label{former blocks}
    \dfrac{|A\cap L_{[1+td,(t+1)d]}|}{|L_{[1+td,(t+1)d]}|}\leqslant 1-\dfrac{1}{p^d-1}.
\end{equation}

Since $0=p^l\notin A$ and $qd+1\geqslant l-d+2$ we have
\begin{equation}\label{last block}
    \dfrac{|A\cap L_{[qd+1,l+1]}|}{|L_{[qd+1,l+1]}|}\leqslant 1-\dfrac{1}{|L_{[qd+1,l+1]}|}\leqslant 1-\dfrac{1}{p^{d-1}} \leqslant 1-\dfrac{1}{p^d-1}.
\end{equation}

Finally we combine \eqref{former blocks} and \eqref{last block} to draw the conclusion that in each block $A$ has a density less than $1-1/(p^d-1)$, therefore $|A|\leqslant (1-1/(p^d-1))N$.

\section*{Acknowledgements}
The author thanks Yifan Jing and Jiaao Li for discussion and guidance.

\bibliographystyle{amsplain}
\bibliography{references}

\end{document}